\documentclass[12pt]{amsart}
\usepackage{palatino}
\usepackage{amsfonts}
\usepackage{amssymb}
\usepackage{amscd}
\usepackage{hyperref}
\usepackage{graphicx}
\usepackage{tikz}
\usepackage{amsmath,amsfonts,amssymb}
\usetikzlibrary{arrows}
\usepackage{color}

\hypersetup{
    colorlinks=true,
    linkcolor=mediumelectricblue,
    citecolor=carnelian,
    urlcolor=blue,
    linktoc=all
}

\definecolor{ashgrey}{rgb}{0.7, 0.75, 0.71}
\definecolor{oxfordblue}{rgb}{0.0, 0.13, 0.28}
\definecolor{armygreen}{rgb}{0.29, 0.33, 0.13}
\definecolor{bulgarianrose}{rgb}{0.28, 0.02, 0.03}
\definecolor{carnelian}{rgb}{0.7, 0.11, 0.11}
\definecolor{lapislazuli}{rgb}{0.15, 0.38, 0.61}
\definecolor{mediumelectricblue}{rgb}{0.01, 0.31, 0.59}

\newtheorem{thm}{Theorem}[section]
\newtheorem{cor}[thm]{Corollary}

\newtheorem{prop}[thm]{Proposition}
\theoremstyle{definition}
\newtheorem{dfn}[thm]{Definition}
\theoremstyle{remark}
\newtheorem{rem}[thm]{Remark}
\numberwithin{equation}{subsection}

\newtheorem{Q}[thm]{Question}
\newtheorem{ex}[thm]{Example}

\newtheorem{con}[thm]{Conventions}

\newcommand{\midarrow}{\tikz \draw[-triangle 45] (0,1) -- +(.5,0);}

\newcommand{\al}{\alpha}
\newcommand{\la}{\lambda}
\newcommand{\ka}{\kappa}

\newcommand{\w}{\omega}
\newcommand{\CC}{\mathcal{C}}
\newcommand{\M}{\overline{M}}

\newcommand{\OO}{\mathcal{O}}
\newcommand{\CF}{\mathcal{F}}
\newcommand{\CH}{\mathcal{H}}

\newcommand{\E}{\mathbb{E}}
\newcommand{\F}{\mathbb{F}}
\newcommand{\G}{\mathcal{G}}
\newcommand{\J}{\mathcal{J}}
\newcommand{\LL}{\mathbb{L}}

\newcommand{\QQ}{\mathbb{Q}}

\newcommand{\Z}{\mathbb{Z}}

\newcommand{\D}{\qquad}

\newcommand{\too}{\rightarrow}

\newcommand{\W}{\mathcal{W}}

\begin{document}

\title[Tautological classes on the moduli space $\CH_{g,n}^{rt}$]
{Tautological classes on the moduli space of hyperelliptic curves with rational tails}
\author[M. Tavakol]{Mehdi Tavakol}
   \address{Korteweg de Vries Instituut voor Wiskunde Universiteit van Amsterdam}
   \email{m.tavakol@uva.nl}
   
\begin{abstract}
We study tautological classes on the moduli space of stable $n$-pointed hyperelliptic curves of genus $g$ with rational tails.
Our result gives a complete description of tautological relations.
The method is based on the approach of Yin in comparing tautological classes on the moduli of curves 
and the universal Jacobian.
It is proven that all relations come from the Jacobian side.
The intersection pairings are shown to be perfect in all degrees.
We show that the tautological algebra coincides with its image in cohomology via the cycle class map. 
The latter is identified with monodromy invariant classes in cohomology.
The connection with recent conjectures by Pixton is also discussed.
\end{abstract}   

\maketitle

\tableofcontents

\section*{Introduction}
In this article we study tautological classes on the moduli space 
$\CH_{g,n}^{rt}$ of stable 
$n$-pointed hyperelliptic curves of genus $g$ with rational tails. 
Tautological classes are natural algebraic cycles reflecting the nature of the generic object parameterized by the moduli space.
The set of generators consists of an explicit collection of cycles.
In particular, tautological groups are finite dimensional vector spaces.
This distinguishes remarkably the tautological ring from the out of reach space of all algebraic cycles. 
A basic question regarding tautological algebras is to give a meaningful class of tautological relations. 

Our strategy in studying tautological classes on $\CH_{g,n}^{rt}$ is to look at the fibers of the projection 
$\pi:\CH_{g,n}^{rt} \too \CH_g$.
The reduced fiber of $\pi$ over a moduli point $[X] \in \CH_g$ 
corresponding to a smooth hyperelliptic curve $X$ is the 
Fulton-MacPherson compactification $X[n]$ of the configuration space of $n$ points on $X$.
There is a natural way to view the tautological ring of 
$X[n]$ as an algebra over the tautological ring $R^*(X^n)$ of the cartesian product $X^n$ of $X$. 
Basic relations among tautological classes on $X^n$ are obtained.
These relations are divided into 3 parts:

\begin{itemize}
\item
The vanishing of the Faber-Pandharipande cycle,

\item
The vanishing of the Gross-Schoen cycle,

\item
A relation of degree $g+1$ involving $2g+2$ points.
\end{itemize}

We show that all these relations can be obtained from relations on the universal Jacobian $\J_g$ over the space $\CH_g$. 
Our argument in proving the first two vanishings depends strongly on the fact that we are working with hyperelliptic curves. 
The last relation has a different nature and holds over any family of smooth curves of genus $g$.
We will give two independent arguments to prove this relation.
Yin has pointed out that the degree $g+1$ relation can be obtained from the vanishing of a certain Chow motive as well.

From our results we get a complete description of tautological relations on the moduli space $\CH_{g,n}^{rt}$. 
We will see that the structure of the tautological algebra is determined by studying tautological classes on the fiber $X[n]$:
\begin{thm}\label{X}
Let $X$ be a fixed hyperelliptic curve of genus $g$.
The tautological ring of the moduli space $\CH_{g,n}^{rt}$ is naturally isomorphic to the tautological ring of the fiber $X[n]$.
In particular, the intersection pairings are perfect in all degrees.
\end{thm}

Everything mentioned above concerns tautological classes in Chow.
The Gorenstein property of $R^*(\CH_{g,n}^{rt})$ implies 
the same results in cohomology.
This shows that there is no difference between Chow and cohomology as long as we restrict to tautological classes.
Using a result of Petersen and Tommasi, 
which was our motivation for this project, we prove the following:

\begin{cor}\label{M}
The cycle class map induces an isomorphism between the tautological ring  of 
the moduli space $\CH_{g,n}^{rt}$ in Chow and monodromy invariant classes in cohomology.
\end{cor}

At the end we discuss the connection between the relations on the space $\CH_{g,n}^{rt}$ and Pixton's relations on $\M_{g,n}$.

\begin{con}
We consider algebraic cycles modulo rational equivalence.
All Chow groups are taken with $\QQ$-coefficients. 
\end{con}

\vspace{+10pt}
\noindent{\bf Acknowledgments.}
I am grateful to all my colleagues who made suggestions and corrections on the preliminary version of this note.

I would like to thank Carel Faber, Gerard van der Geer, Richard Hain, 
Robin de Jong, Nicola Pagani, 
Aaron Pixton and Orsola Tommasi for the valuable discussions and their comments.
Thanks to Felix Janda for sending me the notes on Pixton's relations on products of the universal curve, 
answering many questions in that direction and his comments.
Special thanks are due to Qizheng Yin for explaining several aspects 
of the theory developed in his thesis and his comments and corrections.    
Several parts of this research was carried out during my stay at the 
Max-Planck-Institut f\"ur Mathematik at Bonn in 2013.
This research was completed in the group of Sergey Shadrin at the KdV Instituut voor Wiskunde
at the university of Amsterdam. 
Thanks to Sergey Shadrin for his interest in this project and his supports.
I would like to thank both institutes for their supports.

\section{Tautological classes on the space of hyperelliptic curves}

Let $\CH_{g,n}^{rt}$ be the space of stable $n$-pointed hyperelliptic curves of genus $g$ with rational tails. 
It is a quasi-projective variety of dimension $2g-1+n$.
It parameterizes objects of the form $(C;x_1 , \dots , x_n)$, 
where $C$ is stable hyperelliptic with $n$ distinct smooth points 
$x_i$ for $i=1, \dots , n$.
We assume that $C$ is a reduced nodal curve of arithmetic genus $g$ with exactly one component of genus $g$. 
For each rational component of the curve the markings and nodes are called special points. 
By the stability condition we require that all rational components have at least 3 special points. 
As a result each object of the corresponding moduli problem has finitely many 
automorphisms and the resulting stack is of Deligne-Mumford type.

Tautological classes on $\CH_{g,n}^{rt}$ are defined as natural algebraic cycles on the moduli space.
Let $\pi:\CC \too \CH_g$ be the universal hyperelliptic curve of genus $g$ and denote by $\w$ its relative dualizing sheaf.
Its class in the Picard group of $\CC$ is denoted by $K$.
Denote by $\CC^n$ the $n$-fold fiber product of $\CC$ over $\CH_g$.  
We define $K_i:=\pi_i^*(K) \in A^1(\CC^n)$, 
where $\pi_i:\CC^n \too \CC$ is the projection onto the $i^{th}$ factor for every $1 \leq i \leq n$.
Its pull-back via the contraction map $\CH_{g,n}^{rt} \too \CC^n$ is denoted by the same letter.  
Tautological classes on $\CH_{g,n}^{rt}$ come from the classes $K_i$ and those supported on the boundary of the partial compactification $\CH_{g,n}^{rt}$. 
Recall that for each subset $I$ of the marking set $\{1, \dots , n\}$ 
having at least 2 elements there is a boundary divisor class $D_I$ in Pic$(\CH_{g,n}^{rt})$.
It corresponds to those nodal curves having 2 components.
According to our definition one component is of genus $g$ and the other component is rational. 
The index set $I$ refers to the markings on the rational component.
We now define the tautological ring of the moduli space:

\begin{dfn}
The tautological ring of $\CH_{g,n}^{rt}$ is the $\QQ$-subalgebra of the rational Chow ring 
$A^*(\CH_{g,n}^{rt})$ of $\CH_{g,n}^{rt}$ generated by the divisor classes 
$K_i$ for $i=1, \dots, n$ and the boundary divisors $D_I$ for subsets $I \subseteq \{1, \dots, n\}$ with at least 2 elements. 
\end{dfn}

\begin{rem}
There is another, equivalent, way to define tautological classes.
We consider the system of Chow rings $A^*(\CH_{g,n}^{rt})$ for all stable $(g,n)$. 
The system of tautological rings $R^*(\CH_{g,n}^{rt})$ are defined as the smallest collection of 
$\QQ$-subalgebras of the Chow rings $A^*(\CH_{g,n}^{rt})$ 
having the identity element and stable under all natural maps among these spaces. 
It is straightforward to see that all classes we defined above belong 
to the tautological ring with this definition and they are generators. 
\end{rem}

\begin{rem}
Recall that the line
bundle $\LL_i$, whose fiber over the moduli point 
$(C;x_1, \dots, x_n)$ is the cotangent space of $C$ at $x_i$, gives tautological classes.
Its first Chern class is called the $\psi_i$-class.
It is related to the classes considered before via the following equality:
$$\psi_i=K_i+\sum_{i \in I} D_I.$$
Other natural classes on the moduli space $\CH_g$ are kappa classes.
Recall that the kappa class $\ka_i$ is defined as $\pi_*(K^{i+1})$.
Notice that the kappa class $\ka_i$ vanishes when $i>0$ since $\CH_g$ has trivial Chow groups. 
\end{rem}

We study the connection between tautological classes on moduli of curves and the universal Jacobian.
There are many natural maps from a curve into its Jacobian.
We find it more convenient to use a Weierstra{\ss}  point on the curve to define such a map.
We work with the moduli space $\W_g$ of Weierstra{\ss} pointed hyperellipctic curves of genus $g$. 
This space parameterizes objects of the form $(C,p)$, 
where $C$ is a hyperelliptic curve of genus $g$ and $p$ is a Weierstra\ss \ point on $C$.
The universal family $\pi:\CC \too \W_g$ admits a section $s: \W_g \too \CC$.
It associates the Weierstra{\ss} point $p \in C$ to the pair $(C,p)$.
The space $\W_g$ is a finite cover of $\CH_g$ of degree $2g+2$.
In a similar way we consider a pointed version of this moduli space and define the space $\W_{g,n}^{rt}$.
In an analogous manner we define tautological rings of the moduli spaces parameterizing  Weierstra\ss \ pointed curves.
In each case there is a map to a moduli of curves by ignoring the Weierstra\ss \ point.
Tautological classes are defined as pull-backs of tautological classes on the moduli of curves via these maps.

\section{Tautological classes on the universal Jacobian}

In this section we review basic notions about tautological classes on the universal Jacobian. 
Relations among these classes give interesting results on moduli of curves. 
The tautological ring of a fixed Jacobian variety is defined by Beauville. 
In \cite{B} he studies tautological classes on the Jacobian of a fixed curve under algebraic equivalence. 
The idea is to consider the class of a curve of genus $g$ 
inside its Jacobian and apply all natural operators to it induced from 
the group structure on the Jacobian and the intersection product in the Chow ring. 
He shows that the resulting algebra becomes stable under the Fourier transform.
In fact, if one applies the Fourier transform to the class of the curve, 
all components in different degrees belong to the tautological algebra.
Beauville shows that these components yield a set of generators with $g-1$ elements.
When the curve admits a degree $d$ map into the projective line the tautological ring is generated by $d-1$ elements. 
 
Let $\pi: \CC \too S$ be a family of smooth curves of genus $g > 0$ which admits a section $s: S \too \CC$. 
Denote by $\J_g:=\text{Pic}^0(\CC/S)$ the relative Picard scheme of divisors of degree zero. 
It is an abelian scheme over the base $S$ of relative dimension $g$.
The section $s$ induces an injection $\iota: \CC \too \J_g$ from $\CC$ into the universal Jacobian $\J_g$.
The geometric point $x$ on a curve $C$ is sent to the line bundle $\OO_C(x-s)$ via the morphism $\iota$. 
The abelian scheme $\J_g$ is equipped with the Beauville decomposition defined in \cite{B}. 
Components of this decomosition are eigenspaces of the natural maps corresponding to multiplication with integers.
More precisely, for an integer $k$ consider the associated endomorphism on $\J_g$.
The subgroup $A^i_{(j)}(\J_g)$ is defined as all degree $i$ classes on which the morphism $k^*$ acts via multiplication with $k^{2i-j}$.
Equivalently, the action of the morphism $k_*$ on $A^i_{(j)}(\J_g)$ is multiplication by $k^{2g-2i+j}$. 
The Beauville decomposition has the following form:
$$A^*(\J_g)=\oplus_{i,j} A_{(i,j)}(\J_g),$$
where $A_{(i,j)}(\J_g):=A_{(j)}^{\frac{i+j}{2}}(\J_g)$.

Beside the intersection product on Chow groups there is another multiplication on $A^*(\J_g)$.
The Pontryagin product $*$ defined in terms of the addition 
$$\mu: \J_g \times_S \J_g \too \J_g$$ 
on $\J_g$.
Let $\pi_1,\pi_2: \J_g \times_S \J_g \too \J_g$ be the natural projections and $x,y$ be elements of the Chow ring of $\J_g$.
The Pontryagin product $x*y$ of $x$ and $y$ is defined as $\mu_*(\pi_1^* x \cdot \pi_2^* y)$.

The universal theta divisor $\theta$ trivialized along the zero section 
is defined in the rational Picard group of $\J_g$.
It defines a principal polarization on $\J_g$. 
The first Chern class $l$ of the Poincar{\'e} bundle $\mathcal{P}$ is defined as $\pi_1^* \theta+\pi_2^* \theta-\mu^* \theta$.
The Fourier Mukai transform $\CF$ gives an isomorphism between $(A^* (\J_g),.)$ and $(A^*(\J_g),*)$.
It is defined as follows:
$$\CF(x)=\pi_{2,*}(\pi_1^* x \cdot \exp(l)).$$ 

We now recall the definition of the tautological ring of $\J_g$ from \cite{Y1}.
It is defined as the smallest $\QQ$-subalgebra of the Chow ring $A^*(\J_g)$ which contains the class of $\CC$ and is stable under the Fourier transform and all maps $k^*$ for integers $k$. 
It follows that for an integer $k$ it becomes stable under $k_*$ as well.
From this definition we get infinitely many tautological classes.
But one can see that the tautological algebra is finitely generated.
In particular, it has finite dimensions in each degree.
The generators are expressed in terms of the components of the curve class in the Beauville decomposition.
Define the following classes:

$$p_{i,j}:=\CF \left(\theta^{\frac{j-i+2}{2}} \cdot [\CC]_{(j)} \right) \in A_{(i,j)}(\J_g).$$

We have that $p_{2,0}=-\theta$ and $p_{0,0}=g[\J_g]$. 
The class $p_{i,j}$ vanishes for $i<0$ or $j<0$ or $j>2g-2$.
The tautological class $\Psi$ comes from the class $K$ of relative dualizing sheaf $\omega_{\pi}$ of 
$\pi:\CC \too S$.
It is defined as $$\Psi:=s^*(K).$$

It is proven in \cite{Y1} that the tautological ring of $\J_g$ is generated by the classes $p_{i,j}$ and $\Psi$.
A crucial feature of the tautological ring of the universal Jacobian is the Lefschetz decomposition.
In the classic case the $\mathfrak{sl}_2$ action on Chow groups of an abelian variety was studied by K{\"u}nnemann \cite{KU}.
Polishchuk \cite{PO1} has studied the $\mathfrak{sl}_2$ action for abelian schemes. 
We follow the standard convention that $\mathfrak{sl}_2$ is generated by elements $e,f,h$ satisfying: 
$$[e,f]=h \D [h,e]=2e, \D [h,f]=-2f.$$
In this notation the action of $\mathfrak{sl}_2$ on Chow groups of $\J_g$ is defined as
$$e: A_{(j)}^i(\J_g) \too A_{(j)}^{i+1}(\J_g) \D x \too -\theta \cdot x,$$
$$f: A_{(j)}^i(\J_g) \too A_{(j)}^{i-1}(\J_g) \D x \too -\frac{\theta^{g-1}}{(g-1)!} * x,$$
$$h: A_{(j)}^i(\J_g) \too A_{(j)}^i(\J_g) \D x \too -(2i-j-g) x,$$

The operators $e,h$ have simple forms. 
The operator $f$ is given by the following differential operator:

$$\mathcal{D}=\frac{1}{2} \sum_{i,j,k,l} \left( \Psi p_{i-1,j-1}p_{k-1,l-1}- \binom{i+k-2}{i-1}p_{i+k-2,j+l} \right) \partial p_{i,j} \partial p_{k,l}+\sum_{i,j} p_{i-2,j}\partial p_{i,j}.$$

The differential operator $\mathcal{D}$ is a powerful tool to produce tautological relations. 
The idea is to start from obvious relations and apply the operator $\mathcal{D}$ to it several times. 
This procedure yields a large class of tautological relations.
A surprising fact is that one can get highly non-trivial relations from this method. 

Yin \cite{Y1} studies tautological classes on the universal Jacobian in his recent thesis.
There are basic relations among tautological classes coming from the $\mathfrak{sl}_2$ action on its Chow ring.
Interpreting these relations on the Jacobian side gives an interesting class on relations on moduli of curves.
All tautological relations on the universal curve $\CC_g$ for $g \leq 19$ and on $M_g$ for $g \leq 23$ are recovered using this method.

As another application we will see that all components of the curve class in positive degrees vanish for families of hypereliptic curves.
These vanishings will be used in this article to find all tautological relations on $\CH_{g,n}^{rt}$.
Let $\pi:\CC \too \W_g$ be the universal curve over the space $\W_g$ of Weierstra{\ss} pointed hyperelliptic curves of genus $g$.
Geometric points of $\CC$ correspond to objects of the form $(C,p,x)$
where $C$ is a smooth hyperelliptic curve of genus $g$, $p$ is a Weierestra{\ss} point of $C$ and $x \in C$ is arbitrary.
The degree zero divisor $x-p$ belongs to the Jacobian of $C$. 
This association defines the map
$$\phi: \CC \too \J_g.$$
The $i^{th}$ component of the Beauville decomposition of image of $\CC$ via $\phi$ is denoted by $\CC_{(i)}$ as usual.

\begin{prop}\label{J}
Let $\CC$ and $\phi: \CC \too \J_g$ be as above.
The component $\CC_{(i)}$ vanishes for all $i>0$.
\end{prop}

\begin{proof}
The proof is known to the expert.
For example in \cite{B} Beauville proves a similar statement when $C$ is a fixed hyperelliptic curve. 
He shows that the component $C_{(i)}$ is algebraically equivalent to zero when $i>0$.
Here we want to prove the same vanishings for families of hyperelliptic curves under rational equivalence.

Notice that all components $\CC_{(i)}$ vanish when $i$ is odd.
To see this consider the Ceresa cycle $[\CC]-(-1)^*[\CC]$.
This class is zero according to our definition of the morphism $\phi$.
This shows the vanishing of $\CC_{(i)}$ for odd $i$.
We now use this fact to show the vanishing of other components with positive indices.
Equivalently, we show that the class $p_{i+2,i}=\CF(\CC_{(i)})$ is zero when $i \geq 1$.
The vanishing of $p_{3,1}$ is immediate since we know that $\CC_{(1)}=0$.
Consider the following equations:
$$0=\mathcal{D}(p_{3,1}p_{i+1,i-1})=\Psi p_{2,0} p_{i,i-2}-\binom{i+2}{2}p_{i+2,i}+p_{i+1,i-1}p_{1,1}+p_{3,1}p_{i-1,i-1},$$
which is the same as 
$-\binom{i+2}{2}p_{i+2,i}$.
Notice that the class $\Psi$ 
vanishes since the Picard group of $\W_g$ is trivial. 
This proves the claim since the coefficient of $p_{i+2,i}$ is not zero.
\end{proof}

\begin{rem}
In \cite{B} Beauville shows that the tautological ring of the Jacobian of a fixed hyperelliptic curve of genus $g$ under algebraic equivalence is isomorphic to $\QQ[\theta]/(\theta^{g+1})$.
The vanishing proved above gives the same presentation for the tautological ring of $\J_g$ over $\CH_g$ under rational equivalence.
\end{rem}

\begin{cor}\label{K}
Let $\pi:\CC \too \W_g$ and $s: \W_g \too \CC$ be as before.
We have the relation $K=(2g-2)s$ in Pic$(\CC)$.
\end{cor}

\begin{proof}
Consider the map $\phi: \CC \too \J_g$ defined before.
We will show that the desired relation follows from the vanishing of the divisor $p_{1,1}$.
We need to calculate the pull-back of $p_{1,1}$ to $\CC$ via $\phi$.
The method for calculating the pull-back of tautological classes on the universal 
Jacobian to families of curves is explained in the thesis of Yin.
We briefly recall the procedure from \cite{Y1}.
Let $\pi_1,\pi_2$ be the projections onto the factors of 
$\CC^2:=\CC \times_{\W_g} \CC$.
The section $s: \W_g \too \CC$ induces two sections 
$s_1,s_2: \W_g \too \CC^2$.
Denote by $d_{1,2}$ the class of the diagonal inside $\CC^2$.
According to the definition the class $\phi^* (p_{1,1})$ is equal to 
the degree one component of the following expression:
$$\phi^* \left(\CF(\theta \cdot [\phi_* \CC]) \right).$$

An argument based on chasing through cartesian squares shows that the expression above is the same as:

$$\pi_{2,*} \left( \pi_1^* \left( \left(s+\frac{K}{2}+ \frac{\Psi}{2} \right) \cdot \exp(-2 s) \right) \cdot 
\exp \left(d_{1,2} \right)  \right) \cdot \exp\left(-s\right).$$

We therefore have the following:

$$\phi^*(p_{1,1})=\frac{1}{2}K-(g-1)s-(g-\frac{1}{2})\Psi.$$
The result follows since the classes $p_{1,1}$ and $\Psi$ vanish.
\end{proof}

\section{The Faber-Pandharipande cycle}

There is a natural way to define the tautological ring $R^*(C^n) \subset A^*(C^n)$ of products of a fixed smooth curve $C$ for positive integers $n$. 
It is the $\QQ$-algebra generated by canonical classes and diagonals.
More precisely, denote by $K$ the canonical class on the curve $C$ and consider the natural projections $\pi_i: C^n \too C$ and $\pi_{i,j}:C^n \too C^2$ for $1 \leq i < j \leq n$.
From this collection of maps one gets the divisor classes $K_i:=\pi_i^*(K)$ and $d_{i,j}:=\pi_{i,j}^*(D)$, where $D \subset C^2$ is the diagonal class.
The tautological ring $R^*(C^n)$ is defined to be the $\QQ$-subalgebra of $A^*(C^n)$ generated by $K_i,d_{i,j}$. 
One could simply restrict tautological cycles on the product $\CC_g^n$ of the universal curve $\CC_g$ of genus $g$
to the fiber $C^n$ and recover the same set of generators.

Faber and Pandharipande in an unpublished work have studied this ring in cohomology. 
From their analysis one gets a complete description of this ring.
In particular, there is an explicit presentation of all relations.
The resulting algebra becomes Gorenstein.

It is natural to ask whether 
$R^*(C^n) \subset A^*(C^n)$ 
has the Gorenstein property as well.
The situation becomes difficult already when we consider the surface $C^2$.
There are 4 cycles 
$K_1K_2, K_1d_{1,2},K_2d_{1,2},d_{1,2}^2$ 
in degree 2 with the following relations:
$$d_{1,2}^2=-K_1 d_{1,2}=-K_2d_{1,2}.$$
The proportionality 
$K_1K_2=(2g-2)K_1d_{1,2}$ 
holds in cohomology.
One wonders whether this relation is true in Chow as well.
This was shown by Faber and Pandharipande for $g \leq 3$. 
Green and Griffiths \cite{GG} study the zero cycle 
$K_1K_2-(2g-2)K_1d_{1,2}$ 
for generic curves defined over complex numbers.
Their Hodge theoretic analysis is based on an infinitesimal invariant.
In particular, they show that the Faber-Pandharipande cycle doesn't vanish when $C$ is a generic curve of $g \geq 4$. 
In \cite{Y2} Yin shows that the same statement is true in arbitrary characteristic.
The idea of the proof is to write the Faber-Pandharipande cycle 
as the pull-back of a tautological class from the Jacobian of the curve.
He observes that the corresponding tautological class on the Jacobian doesn't vanish for a generic curve of genus $\geq 4$.
Yin proves that the same is true for its pull-back.
This shows the non triviality of this cycle for such curves. 
We will prove the vanishing of the Faber-Pandharipande cycle on the locus of hyperelliptic curves with the same idea:

\begin{prop} \label{FP}
Let $\pi:\CC \too \CH_g$ 
be the universal hyperelliptic curve of genus $g$.
The cycle $K_1K_2-(2g-2)K_1d_{1,2}$ vanishes on $\CC^2$.
\end{prop}

\begin{proof}

Let $\pi:\CC \too \W_g$ be the universal curve over the space $\W_g$ together with the section $s: \W_g \too \CC$.
Denote by $s_1,s_2$ the induced sections from $\W_g$ to the space
$\CC \times_{\W_g} \CC$.
We have the relations 
$$K_i=(2g-2)s_i, \D \text{for} \ i=1,2$$ 
from the statement proven in Corollary \ref{K}.
This gives the desired vanishing since $K_1K_2=(2g-2)^2s_1s_2$ and 
$K_1d_{1,2}=(2g-2)s_1s_2$.

There is another way to see this:
Consider the following map 
$$\phi_2: \CC \times_{\W_g} \CC \too \J_g.$$

Notice that a geometric point on the space 
$\CC \times_{\W_g} \CC$ has the form 
$(C,p,x,y)$, where $C$ is a hyperelliptic curve with a Weierstra{\ss} point $p$ and $x,y \in C$ are arbitrary. 
The image of this point under $\phi_2$ is the divisor $x+y-2p$ on the Jacobian of $C$. 
After calculating the pull-back of the classes $p_{1,3},p_{2,2}$ to 
$\CC \times_{\W_g} \CC$ we obtain the following relation:

$$\phi_2^*\left(4g p_{2,2}+(4g+6)p_{1,3}\right)=
-\left(\frac{g}{g-1} \right)^2 \left(K_1K_2-(2g-2)K_1d_{1,2}\right).$$

The result follows from the vanishing of $p_{1,3},p_{2,2}$ on the universal Jacobian proved in Proposition \ref{J}.
\end{proof}

The vanishing of the Faber-Pandharipande cycle can be used to show another vanishing on the universal hyperelliptic curve:

\begin{cor} \label{K2}
The cycle $K_1^2$ vanishes on the universal curve $\pi:\CC \too \CH_g$.
\end{cor}

\begin{proof}
We have the relations $$K_1K_2=(2g-2)K_1d_{1,2}=(2g-2)K_2d_{1,2}$$ 
on the space $\CC \times_{\CH_g} \CC$. 
Intersect these relations with the divisor classes $K_1,K_2$ and compute their push-forwards to $\CC$ under the natural projection 
$\CC \times_{\CH_g} \CC \too \CC$, 
onto the first factor. 
The vanishing of $K_1^2$ follows. 
\end{proof}

\begin{rem}
In \cite{Y2} the Faber-Pandharipande cycle is shown to be the pull-back of
the class 
$W:=2p_{1,1}^2-(4g-4)p_{2,2}$,
for a fixed curve of genus $g$.
It is possible to prove Proposition \ref{FP} and Corollary \ref{K2} using the class $W$ with a similar method.
The only difference is that the pull-back of $W$ has extra terms involving $K_1^2$ and $K_2^2$, which both vanish after all.
\end{rem}

\section{The Gross-Schoen cycle}

In \cite{GS} Gross and Schoen considered a smooth and projective curve $X$ defined over a field $k$ together with a $k$-rational point $p$.
The codimension 2 cycle $\Delta_p$ on the product $X^3$ is defined in terms of the diagonal classes and the point $p$.
The authors call this class \emph{the modified diagonal cycle} and study some of its properties.
The basic fact about $\Delta_p$ is that it vanishes in cohomology.
It is proven that the class $\Delta_p$ in the second Griffiths group $Gr^2(X^3):=
A^2_{hom}(X^3)/A^2_{alg}(X^3)$,
measuring homologically trivial cycles modulo algebraically trivial cycles, is independent of the choice of the point $p$.
When $X$ is a rational curve, an elliptic curve or a hyperelliptic curve the cycle $\Delta_p$ is shown to be zero in Chow.
In the first two cases the point $p$ can be arbitrary but for hyperelliptic curves it has to be a Weierstra{\ss} point.
In this article we want to show the same result in the relative setting.

Let us recall the definition of the Gross-Schoen cycle in the classical case. 
Let $X$ be a smooth curve with a point $p$ as above.
Consider the following subvarieties:

$$\Delta_1=\{(x,p,p): x \in X \}, \D 
\Delta_2=\{(p,x,p): x \in X \},$$ 
$$\Delta_3=\{(p,p,x): x \in X \}, \D
\Delta_{1,2}=\{(x,x,p): x \in X \},$$
$$\Delta_{1,3}=\{(x,p,x): x \in X \}, \D
\Delta_{2,3}=\{(p,x,x): x \in X \},$$
$$\Delta_{1,2,3}=\{(x,x,x): x \in X \}.$$

The degree 2 cycle $\Delta_p$ is defined on the product $X^3$ as follows:

$$\Delta_p=\Delta_{1,2,3}-\Delta_{1,2}-\Delta_{1,3}-\Delta_{2,3}+\Delta_1+
\Delta_2+\Delta_3.$$

There is another version of this cycle with respect to the canonical class.
It has the following form:
$$\Delta_K=\Delta_{1,2,3}-\frac{1}{2g-2}\left(K_1d_{2,3}+K_2d_{1,3}+K_3d_{1,2}\right)+\frac{1}{(2g-2)^2}\left(K_1K_2+K_1K_3+K_2K_3\right),$$
where $d_{i,j}$ is the diagonal class $x_i=x_j$ as usual.
Recall that a curve $C$ has a subcanonical point $p$ if the equality $K=(2g-2)p$ holds.
In this situation the classes $\Delta_K$ and $\Delta_p$ defined above coincide with each other. 
Notice that the cycle $\Delta_K$ is symmetric.

\begin{prop}
The cycle $\Delta_K$ vanishes on the locus of hyperelliptic curves.
\end{prop}

\begin{proof}
Let $\pi: \CC \too \W_g$ be the family of Weierstra{\ss} pointed hyperelliptic curves of genus $g$ as before.
We define a map 
$$\phi_3: \CC \times_{\W_g} \CC \times_{\W_g} \CC \too \J_g.$$

For a pointed curve $(C,p)$ and points $x_1,x_2,x_3 \in C$ the associated divisor on the Jacobian of $C$ is the divisor $x_1+x_2+x_3-3p$. 
The cycle $\Delta_K$ is the pull back of the following class via $\phi_3$:
\begin{equation}\tag{1}\label{GS}
p_{3,1}-\frac{1}{g-1}p_{2,0}p_{1,1}+\frac{2g}{g-1}p_{2,2}-
\frac{2g-3}{2(g-1)^2}p_{1,1}^2.
\end{equation}

This was proven by Yin in \cite{Y1} for a fixed curve.
A computation similar to \ref{K} shows that this formula stays valid over the base $\W_g$ as well. 
The vanishing of the cycle $\Delta_K$ follows from Proposition \ref{J}.

\end{proof}

\begin{rem}
In \cite{Z} Zhang studies the connection between the triviality of the Gross-Schoen cycle and the Ceresa cycle in the Chow ring of $X^3$ for a fixed curve $X$.
From his result one can see their equivalence assuming the triviality of the Faber-Pandharipande cycle. 
The vanishing of the Ceresa cycle for families  of Weierstra{\ss} pointed hyperelliptic curves
is obvious from its definition.
It would be interesting to see whether the result proven above can be obtained from this vanishing by the approach of Zhang.
That might give insight into the following natural question about the torsion cycles found in this article:
\begin{Q}
What are the orders of the Faber-Pandharipande and Gross-Schoen cycle in the integral Chow ring of the moduli space of pointed hyperelliptic curves?
\end{Q}
\end{rem}

Modified diagonals can be defined in a more general settings.
The following formulation is due to O'Grady.
Let $X$ be any $d$-dimensional algebraic variety over a field $k$ and $p \in X(k)$ be a $k$-rational point.
The modified cycle $\Gamma^m(X,p)$ can be defined on the product $X^n$.
For any subset $I$ of $\{1, \dots, n\}$ let
$$\Delta_I^n(X,p):=\{(x_1,\dots,x_n): x_i=x_j \ \text{if} \ i,j \in I \text{and} \ x_i=p  \ \text{if} \ i \notin I\}.$$
The $n^{th}$ modified cycle associated to the point $p$ is the $d$-cycle on $X^n$ given by
$$\Gamma^n(X,p):=\sum_{\emptyset \neq I \subset \{1, \dots, n\}} (-1)^{n-|I|} \Delta_I^n(X,p).$$
In \cite{O} it is conjectured that for a hyperk{\"a}hler variety of dimension $d=2n$ there exists a point $p \in X$ such that the cycle $\Gamma^{2n+1}(X,p)$ vanishes.
It is also conjectured that the modified diagonal 
$\Gamma^{2g+1}(A,p)$ vanishes for a point $p$ on an abelian variety of dimension $g$.
A recent result of Moonen and Yin \cite{MY} establishes the second conjecture. 
In \cite{MY2} the same authors among other things give a motivic description of modified diagonals.

\section{The degree $g+1$ relation}\label{g+1}

We have proved that the Faber-Pandharipande cycle and Gross-Schoen cycle vanish on families of hyperelliptic curves.
In this section we obtain a degree $g+1$ relation which plays an essential role in our study. We will see that there are two different ways to get this relation.
The first source of this relation is again the universal Jacobian.
We use the formula given by Grushevsky and Zakharov for the pull-back of the theta divisor to the moduli of curves.
The second method is based on studying linear systems for generic curves of genus $g$. This method produces relations in more general settings.
Restricting to the locus of hyperelliptic curves we obtain a tautological relation of degree $g+1$ involving $2g+2$ points.

\subsection{Relations coming from the theta divisor}

The geometry of the theta divisor on the universal Jacobian gives a very simple way to prove our relation.
The pull-back of the theta divisor to the space of curves is investigated by several authors.
Hain \cite{H} gives a formula for this class in terms of standard boundary cycles in cohomology. 
Hain uses this formula to compute the pull-back of the zero section of the universal Jacobian to the space of curves of compact type. 
Hain's formula answers Eliashberg's question.
Analogue results are proven by M{\"u}ller
\cite{MUL}.
Grushevsky and Zakharov \cite{GZ1}, \cite{GZ2} give a formula for the pull-back of the theta divisor to the spaces $M_{g,n}^{ct}$ classifying pointed curves of compact type and the space $\M_{g,n}$ of stable pointed curves.
Here we follow the notation in \cite{GZ1}.
Let $\J_g$ be the universal Jacobian of degree zero divisors over $M_g^{ct}$.
We denote its pull-back under the natural projection $M_{g,n}^{ct} \too M_g^{ct}$ by the same letter.
Consider a collection $\mathbf{d}=(d_1, \dots, d_n) \in \Z^n$ of integers satisfying $\sum_{i=1}^nd_i=0$. 
For any moduli point 
$$(C;x_1 , \dots, x_n) \in M_{g,n}^{ct}$$
one gets a degree zero divisor $\sum_{i=1}^n d_ix_i$ on the Jacobian of $C$.
This association defines a map
$$s_{\mathbf{d}}:M_{g,n}^{ct} \too \J_g.$$
Let $\theta$ be the universal symmetric theta divisor trivialized along the zero section as before. 
In \cite{GZ1} Grushevsky and Zakharov compute the pull-back 
$s_{\mathbf{d}}^*(\theta)$ in terms of standard divisor classes on $M_{g,n}^{ct}$. 
We recall the definition of the divisor class $\Delta_{h,I}$ for $0 \leq h \leq g$ and a subset $I$ of the marking set $\{1, \dots, n\}$.
The generic point on this divisor corresponds to a singular curve having 2 irreducible components.
One of the components has genus $h$ whose set of markings is $I$.
For any such subset $I$ the number $d_I$ is defined as the sum $\sum_{i \in I}d_i$.

\begin{thm}
For deg $\mathbf{d}=0$, the class $s_{\mathbf{d}}^*(\theta) \in \text{Pic}_{\QQ}(M_{g,n}^{ct})$ 
of the pull-back of the universal symmetric theta divisor trivialized along the zero section is equal to
$$s_{\mathbf{d}}^*(\theta)=\frac{1}{2}\sum_{i=1}^n d_i^2K_i-\frac{1}{2}\sum_{I \subseteq \{1, \dots , n\}} (d_I^2-\sum_{i \in I} d_i^2) \Delta_{0,I}-
\frac{1}{2} \sum_{h>0, I \subset \{1, \dots, n\}} d_I^2\Delta_{h,I}$$ 
\end{thm}

\begin{rem}
In \cite{GZ1} a similar formula is proven when deg $\mathbf{d}=g-1$.
For hyperelliptic curves these relations give equivalent results.
\end{rem}

The vanishing of the class $\theta^{g+1}$ in the Chow ring of the universal Jacobian $\J_g$ gives a relation among tautological classes on $M_{g,n}^{ct}$.
We restrict this relation to the locus of hyperelliptic curves and get a relation on $\CH_{g,n}^{rt}$.
One can show that the relations coming from the vanishing of the class $\theta^{g+1}$ follow from the relations found on $\CH_{g,n}^{rt}$ for $n \leq 3$ as long as $n <2g+2$.
When $n=2g+2$ one gets \emph{one} new relation.
More precisely, for all choices of parameters $d_i$, $i=1, \dots, 2g+2$ the resulting relations on $\CH_{g,2g+2}^{rt}$ are multiples of each other up to a linear combination of the relations involving $n \leq 3$ points. 
As we will see in Section \ref{Product} there will be no new relations afterwards. 

\subsection{Relations from higher jets of differentials}

The method is similar to the method introduced by Faber in \cite{F2} with slight changes. 
This gives tautological relations on products $\CC_g^n$ of the universal curve over $M_g$. 
The resulting relation holds for a general family of curves.

Let $\pi: \CC_g \too M_g$ 
be the universal curve of genus $g$ with the relative dualizing sheaf $\w$. 
We also make the usual convention $K:=c_1(\w)$.
The $n$-fold fibered product of the curve $\CC_g$ over the $M_g$ is denoted by $\CC_g^n$. 
We consider two natural locally free sheaves on this space.
Let $\pi:\CC_g^{n+1} \too \CC_g^n$ be the projection onto the first $n$ factors.
Its relative dualizing sheaf is denoted by $\w_{n+1}$.
The sum of the diagonal classes $d_{i,n+1}$ on $\CC_g^{n+1}$ defines the divisor class $\Delta_n$:
$$\Delta_n=\sum_{i=1}^n d_{i,n+1}.$$
The locally free sheaf $\E_m$ defined for every $m \geq 0$ as follows:
$$\E_m:=\pi_*(\w^{\otimes m}).$$
This is the usual Hodge bundle of rank $g$ when $m=1$.
The fiber of $\E_m$ at a point $(C;x_1 , \dots, x_n)$ is the vector space
$H^0(C,\w_C^{\otimes m})$, 
where $\w_C$ is the dualizing sheaf of the curve $C$.
For $m>1$ it is of rank $(2m-1)(g-1)$.
Another natural bundle is obtained from evaluating differential forms on divisors.
We define the following locally free sheaf of rank $n$ on $\CC_g^n$:
$$\F_{m,n}:=\pi_*\left(\OO_{\Delta_{n+1}} \otimes \w_{n+1}^{\otimes m}\right).$$

The fiber of the sheaf $\F_{m,n}$ at a point $(C;x_1,\dots, x_n)$ is 
$$H^0\left(C,\frac{\w_C^{\otimes m}}{\w_C^{\otimes m}\left(-\sum_{i=1}^n x_i\right)} \right).$$

Consider the natural evaluation map:
$$\phi_{m,n}:\E_m \too \F_{m,n}.$$
For a general $n$ the morphism $\phi_{m,n}$ doesn't behave well.
Its kernel has the fiber 
$$H^0\left(C,\w_C^{\otimes m}\left(-\sum_{i=1}^n x_i\right)\right),$$ 
which depends on the curve $C$ and the points $x_1, \dots, x_n$.
However the situation becomes simpler when $n$ is large enough. 
More precisely, when $n > 2m(g-1)$ the morphism $\phi_{m,n}$ is injective.
The quotient bundle $\F_{m,n}/\E_m$ has rank $n-(2m-1)(g-1)$.
This means that for all $r:=n-2m(g-1) >0$ we get the following vanishing:

\begin{prop}\label{R}
The class $c_{g+r}(\F_{m,n}-\E_m)$ vanishes for all $r >0$. 
\end{prop}

The Chern classes of the bundle $\F_{m,n}$ can be calculated with the same method as in \cite{F2} using 
Grothendieck-Riemann-Roch.
We have the following formula for its total Chern class:

$$c(\mathbb{F}_{m,n})=(1+mK_1)(1+mK_2-\Delta_2) \dots (1+mK_n-\Delta_n).$$

We recover Faber's relations \cite{F2} when $m=1$.
As we will see there are tautological relations on $\CC^n_g$ which don't come from Faber's relations.

As an example let $g=2$ and take $m=3$. 
In this case we get the relation $c_3(\F_{3,7}-\E_3)=0$.
After multiplying this relation with the divisor class $K_7$ on $\CC_2^7$ and pushing it forward to $\CC_2^6$ via the projection $\CC_2^7 \too \CC_2^6$ we get a degree 3 relation involving 6 points. 
This relation was found in \cite{T2} and was used to study the tautological ring of $M_{2,n}^{rt}$.

As another example take $m=2$ for $g >2$. 
From Proposition \ref{R} we get the following relation involving $n=4g-3$ points:
$$c_{g+1}(\F_{2,4g-3}-\E_2)=0.$$

Notice that $n \geq 2g+2$ by our assumption.
From this relation we get a degree $g+1$ relation involving $2g+2$ points.
There are many ways to do this. All give equivalent results for hyperelliptic curves.
Here is one example of doing this:
Multiply the relation above with the monomial $K_{2g+3} \dots K_{4g-3}$ and push it forward to $\CC_g^{2g+2}$.
The resulting relation is symmetric with respect to $2g+2$ markings.

\section{Products of the universal curve over $\CH_g$}
\label{Product} 

In this section we give a description of the tautological rings for products
$\CC^n$ of the universal hyperelliptic curve $\pi:\CC \too \CH_g$.
We will see that the relations found in previous sections can be used to find all tautological relations. This is based on explicit computations of the intersection pairings.
To simplify the computations we work with a different set of generators.

\begin{dfn}
Let $n \geq 1$ be an integer. For every $1 \leq i \leq n$ and $1 \leq i < j \leq n$ define the following classes:
$$a_i:=\frac{1}{2g-2}K_i, \D b_{i,j}:=d_{i,j}-a_i-a_j.$$
\end{dfn}

It is straightforward to see that the collection of elements $a_i,b_{i,j}$ generate the tautological algebra of $\CC^n$. 
The relations we found in previous sections become simpler in terms of these variables.
The relation $K_i^2=0$ translates into $a_i^2=0$. 
The following relations
$$K_1K_2-(2g-2)K_1d_{1,2}=K_1K_2-(2g-2)K_2d_{1,2}=0$$ 
are equivalent to the vanishings $a_ib_{i,j}=b_{i,j}^2+2ga_ia_j=0$. 
The vanishing of the Gross-Schoen cycle is equivalent to $b_{i,j}b_{i,k}-a_ib_{j,k}=0$.
The degree $g+1$ relation comes from the vanishing of the following symmetric expression:

\begin{equation}\tag{2}\label{Re}
\al_g:=\sum_{\mathcal{I}} b_{i_1,i_2} \dots b_{i_{2g+1},i_{2g+2}}.
\end{equation}

Each term of the sum corresponds to a partition $\mathcal{I}$ of $\{1, \dots, 2g+2\}$ into $g+1$ subsets with 2 elements. 

\begin{ex}
While we consider only hyperelliptic curves our presentation works for elliptic curves as well.
In this case the origin of the elliptic curve plays the role of a Weierstra{\ss} point.

Let $\pi: \CC \too M_{1,1}$ be the universal elliptic curve over $M_{1,1}$. 
Geometric points of $M_{1,1}$ are elliptic curves $(C,p)$,
where $C$ is a smooth curve of genus one and $p \in C$ denotes its origin.
The morphism $\pi$ admits a natural section $a: M_{1,1} \too \CC$.
It associates $p \in C$ to the moduli point $(C,p)$. 
The image of the section $a$ is denoted by the same letter.
Consider the $n$-fold fiber product $\CC^n:=\CC \times_{M_{1,1}} \dots \times_{M_{1,1}} \CC$. 
Notice that $\CC^n$ is birational to the moduli space $M_{1,n+1}$!
The divisor class $a$ defines the divisor $a_i$ 
in the Picard group of $\CC^n$ for $1 \leq i \leq n$.
We also have the diagonal class $d_{i,j}$ for $1 \leq i < j \leq n$.
The class $b_{i,j}$ is defined as $b_{i,j}:=d_{i,j}-a_i-a_j$.
In \cite{T1} the vanishing of the Faber-Pandharipande cycle on $\CC^2$ is obtained from a tautological relation on $M_{1,3}^{ct}$.
The vanishing of the Gross-Schoen cycle gives $b_{1,2}b_{1,3}-a_1b_{2,3}=0$.
The connection between this relation and Getzler's relation on $\M_{1,4}$ is explained in \cite{T1}.
The next case deals with $n=4$.
There are 10 generators in degree 1,3 and 21 generators in degree 2.
The intersection matrix of the pairing $R^1(\CC^4) \times R^3(\CC^4)$ is invertible.
This shows that tautological groups are of dimension 10 in degrees 1 and 3.
The resulting intersection matrix for $R^2(\CC^4) \times R^2(\CC^4)$ has the form
$$\begin{pmatrix}
I_6 &  & &  &  \\
 & -2I_{12} &  &  &  \\
   &  & 4 & -2 & -2  \\
   &  & -2 & 4 & -2 \\
   &  & -2 & -2 & 4 \\
\end{pmatrix}.$$

This matrix has rank 20. 
The kernel is one dimensional and corresponds to the relation
$$b_{1,2}b_{3,4}+b_{1,3}b_{2,4}+b_{1,4}b_{2,3}=0.$$
This relation can be obtained from Getzler's relation via a pull-back.
It is proven in \cite{T1} that every relation in the tautological ring of $\CC^n$ follows from these relations.
This was used to find all tautological relations for the moduli space $M_{1,n}^{ct}$.
\end{ex}

\begin{ex}
Let $g=2$ and consider $n=2$.
The vanishing of the Faber-Pandharipande cycle gives the relation $b_{1,2}^2+4a_1a_2=0$. 
This is the restriction of a tautological relation found on $\M_{2,2}$ by Getzler \cite{G2}. 
When $n=3$ the vanishing of the Gross-Schoen cycle corresponds to the relation $b_{1,2}b_{1,3}-a_1b_{2,3}=0$.
This is the restriction of the relation on $\M_{2,3}$ found by Belorousski and Pandharipande \cite{BP}.
The last relation contains 15 terms and has the following form:
$$\sum b_{i_1,i_2}b_{i_3,i_4}b_{i_5,i_6}=0.$$
In \cite{T3} we proved that these relations determine the structure of the tautological ring $R^*(\CC_2^n)$ for every $n$. 
\end{ex}

The relations involving $n \leq 3$ points can be used to generate the tautological group of a given degree with elements of the form
\begin{equation}\label{STC}\tag{3}
v:=\prod_{i \in A(v)} a_i \cdot \prod_{j,k \in B(v)} b_{j,k}, \D A(v) \cap B(v)=\emptyset.
\end{equation}

In this situation such element $v$ is said to be a \emph{standard monomial}.
We define $a(v):=\prod_{i \in A(v)} a_i$ and $b(v):=\prod_{j,k \in B(v)} b_{j,k}$.
What we mentioned before means that the tautological group of $\CC^n$ is generated by standard monomials.
Intersection pairings have a simple form if one works with standard monomials.
There is a natural way to associate a standard monomial $w \in R^{n-k}(\CC^n)$ to every standard monomial $v \in R^k(\CC^n)$.
It is simply defined as the following product
$$w:=\prod_{i \in \{1, \dots, n\} \setminus A(v) \cup B(v)} a_i \cdot b(v).$$ 

We say that $v$ and $w$ are dual to each other and write $w=v^*$.
An elementary argument shows that for standard monomials $v_1,v_2$ of the same degree the product $v_1^* \cdot v_2$ vanishes unless they have the same $B$-parts.
This means that interesting blocks of intersection pairings come from matrices having the following form:
Let $m \geq 1$ be an integer and consider all standard monomials of the
form $b_{i_1,i_2} \dots b_{i_{2m-1},i_{2m}}$.
These belong to the tautological group $R^m(\CC^{2m})$.  
Denote by $R_{2m}$ the $\QQ$-vector space generated by these elements.
The permutation group $S_{2m}$ acts on $R_{2m}$ via its natural action on indices. This makes $R_{2m}$ into a representation of the symmetric group $S_{2m}$. 
The decomposition of $R_{2m}$ into irreducible components has the following form:
$$R_{2m}=\oplus_{\la} V_{\la},$$
where $V_{\la}$ is the representation associated to the partition $\la$ and in this sum $\la$ varies over all partitions having only even components.
Notice that all representations appear with multiplicities 1.
In \cite{HW} similar matrices and their eigenvalues are studied by Hanlon and Wales.
It follows from their result that every representation $V_{\la}$ corresponds to an eigenspace of our intersection matrix.
In particular, $V_{\la}$ is in the kernel if and only if the partition $\la$ has a part of length at least $2g+2$.
An elementary argument shows that such representation $V_{\la}$ is generated by an element which is a linear combination of expressions of the form $\al_g$ given in \eqref{Re}.
This shows that the relations found for $n \leq 3$ points together with the degree $g+1$ relation on $\CC^{2g+2}$ generate all tautological relations.
This completes the description of $R^*(\CC^n)$ for every $n$.
It follows that intersection pairings are perfect in all degrees.

Everything we have said about the tautological ring of the moduli space $\CC^n$ can be restricted to a fixed hyperelliptic curve.

\begin{cor}
Let $X$ be any hyperelliptic curve of genus $g$ and $n$ be an integer.
The intersection pairings between tautological classes on $X^n$ are perfect in all degrees. 
\end{cor}

\section{The Fulton-MacPherson compactification}

In previous sections we found tautological relations on products of the universal hyperelliptic curve $\CC$ over $\CH_g$.
Those relations suffice to determine the structure of the tautological ring of $\CC^n$ for every $n$.
We now want to include singular curves as well and give a description of $R^*(\CH_{g,n}^{rt})$.
To do this goal we notice that the space $\CH_{g,n}^{rt}$ can be seen as the relative Fulton-MacPherson compactification of $\CC^n$ over the base $\CH_g$.
More precisely, consider the natural map
$$\pi: \CH_{g,n}^{rt} \too \CH_g,$$
which forgets all markings on the curve and contracts all rational components.
The reduced fiber of $\pi$ over a moduli point $[X] \in \CH_g$ is the 
Fulton-MacPherson compactification $X[n]$ of $X$.
From this point of view all boundary divisors appearing in the partial compactification $\CH_{g,n}^{rt}$ can be seen as exceptional divisors of blow-ups 
introduced in the process of separating points.
This identifies the tautological ring of $\CH_{g,n}^{rt}$ with an algebra over $R^*(\CC^n)$. 
Basic relations among exceptional divisors are interpreted as trivial relations on the moduli space. 
This gives a simple description of the tautological ring of $\CH_{g,n}^{rt}$
in terms of the results proven in previous sections.

We first recall the definition of the Fulton-Macpherson compactification of a configuration space in the classical settings.
Let $X$ be a nonsingular algebraic variety defined over an algebraically closed field. For an integer $n \geq 1$ consider the product $X^n$.
The configuration space $F(X,n)$ is the open subset of $X^n$ corresponding to all $n$-tuples consisting of $n$ distinct points.
In \cite{FM} Fulton and MacPherson introduce a compactification of $F(X,n)$ inside the following product
$$X^n \times \prod_{|S| \geq 2} \text{Bl}_{\Delta} X^S,$$
where for each subset $S$ of the set $\{1, \dots , n\}$ Bl$_{\Delta} X^S$ 
is the blow-up of the corresponding product $X^S$ along its small diagonal $\Delta$. 
The resulting space is denoted by $X[n]$. 
This construction is symmetric with respect to the action of the symmetric group
permuting $n$ points.
The space $X[n]$ is an irreducible algebraic variety.
The natural map from it to $X^n$ is proper.
Furthermore, the variety $X[n]$ contains $F(X,n)$ as an open subset and the complement is a divisor with normal crossings.  
This compactification has several advantages over the naive candidate $X^n$.
When there are several points in the configuration space approaching to the same point $x \in X$ the resulting limit in $X^n$ is simply $x$.
In this compactification one remembers how fast these points approach each other.
The original construction of Fulton and MacPherson is inductive.
The starting point is $X[1]:=X$. Assuming that $X[n]$ is already constructed we consider the product $X[n] \times X$. There is a sequence of blow-ups along a collection of disjoint codimension two subvarieties which yields $X[n+1]$.

Another equivalent construction of the compactification $X[n]$ is given in \cite{L}.
The construction consists of a symmetric sequence of blow-ups along diagonal classes on $X^n$. 
We consider a diagonal class $X_I$ for each subset $I$ of $\{1, \dots , n\}$ having at least 2 elements.
It is defined as all points in $X^n$ having equal elements for all coordinates corresponding to the set $I$.
In other words it is the inverse of the small diagonal of $X^{|I|}$ via the natural projection from $X^n$ onto $X^{|I|}$.
We first blow up the small diagonal inside the product $X^n$.
In the next step we blow up all diagonals associated with subsets with $n-1$ elements. 
This process continues and at each step we increase the 
dimension of the blow-up centers by one. 
Notice that in our case we assume that $X$ is a curve and it suffices to deal with subsets $I$ having at least 3 elements. 
The class of the exceptional divisor of the blow-up along $X_I$ and its proper transform under later blow-ups is denoted by $D_I$.

\section{The tautological ring of $X[n]$}

In \cite{FM} the authors show that the Chow ring of $X[n]$ can be described as an algebra over the intersection ring $A^*(X^n)$ of the cartesian product $X^n$.
They also give a description of the ideal of relations.
Notice that we chose a different sequence of blow-ups than the original construction given by Fulton and MacPherson. 
As a result our presentation of the Chow ring has a slightly different form.
The formula of Keel \cite{K} plays an essential role in our computations.
To apply this formula we consider a codimension $d$ closed subvariety $Z$ of an algebraic variety $Y$. 
We assume that the restriction map 
$$A^*(Y) \too A^*(Z)$$
is surjective. 
Let $J_{Z/Y}$ be its kernel so that
$$A^*(Z)=\frac{A^*(Y)}{J_{Z/Y}}.$$

Notice that this property holds for all subvarieties occurring in the process of the Fulton-MacPherson compactification. 
By our assumption the Chern class $c_i(N_{Z/Y})$ of the normal bundle of $Z$ is the restriction of an algebraic cycle $a_i \in A^i(Y)$.
We define a \emph{Chern polynomial} for $Z$ to be
$$P_{Z/Y}(t)=t^d+a_1t^{d-1}+ \dots + a_{d-1}t+a_d \in A^*(Y)[t].$$
We require that $a_d=[Z]$.
Other coefficients $a_i$ are well-defined only modulo the ideal $J_{Z/Y}$.
Denote by $\widetilde{Y}$ the blow-up of $Y$ along $Z$ and let $E$ be the class of the exceptional divisor.
In this situation we get a simple description of the Chow ring of $\widetilde{Y}$ in terms of the Chow ring of $Y$ as follows:

\begin{thm}(Keel)
The Chow ring of $\widetilde{Y}$ is given by
$$A^*(\widetilde{Y})=\frac{A^*(Y)[E]}{\left(J_{Z/Y} \cdot E, P_{Z/Y}(-E)\right)}.$$
\end{thm}

This theorem enables us to find a complete description of the Chow ring of $X[n]$ if we know $A^*(X^n)$. 
We need to know a Chern polynomial $P_{Z/Y}$ and the ideal $J_{Z/Y}$ for all subvarieties $Z$ appearing in the construction of the space $X[n]$. 
In our case we are only interested in tautological classes.
There is a natural way to define the tautological ring of $X[n]$.

\begin{dfn}
Let $X$ be a curve and $n$ be a positive integer. 
The tautological ring of the Fulton-MacPherson compactification $X[n]$ of 
the configuration space $F(X,n)$
is defined to be the $\QQ$-subalgebra of the rational Chow ring of $X[n]$ generated by tautological classes on $X^n$ together with the classes of exceptional divisors.
\end{dfn}
 
In \cite{T2} we found the connection between tautological relations on $X[n]$ and the product $X^n$. 
They are divided into 5 classes of relations.
In the following we assume that $Y$ is $X^n$ or the total space of one of its blow-ups in the process of the construction of $X[n]$:

\begin{enumerate}\label{relations}
\item
The most obvious class of relations consists of those coming from the original space $X^n$. Notice that the Chow ring of $X^n$ naturally becomes a subring of the intersection ring of the space $X[n]$.
We use the same letters for cycles in $R^*(X^n)$ and their pull backs to $X[n]$.

\item
Another class of trivial relations comes from the vanishing of products of certain exceptional divisors. 
We have seen that exceptional divisors
correspond to subsets $I$ of $\{1, \dots, n\}$ having at least 3 elements.
Let $D_I$ and $D_J$ be the exceptional divisors associated to the subvarieties $X_I$ and $X_J$ of $X^n$.
Assume that $I \not \subseteq J$, $J \not \subseteq I$ and $I \cap J \neq \emptyset$.
Under these assumptions the proper transforms of the subvarieties $X_I$ and $X_J$ become disjoint in the process of blow-ups.
This happens when we blow up subvarieties corresponding to subsets having 
$\min(|I|,|J|)+1$ elements.
This implies that the product $D_I \cdot D_J$ is zero for such subsets.

\item
For a subvariety $Z$ of $Y$ consider the inclusion map $i:Z \too Y$.
Let $x$ be an element in the kernel of the restriction map $i^*:A^*(Y) \too A^*(Z)$.
Denote by $E_Z$ the class of the exceptional divisor of the blow-up Bl$_Z Y$ of $Y$ along $Z$. We get the relation $x \cdot E_Z=0$ in the Chow ring of Bl$_Z Y$.

\item
Let $V_1,\dots,V_k,Z$ be a collection of blow-up centers.
Assume that the subvarieties $V_1, \dots, V_k$ intersect transversally.
Furthermore, suppose we can write $Z$ as the transversal intersection 
$V_1 \cap \dots \cap V_k \cap W$ for some $W$.
Denote by $E_{\widetilde{V_i}}$ the class of the exceptional divisor of the blow-up along $V_i$.
In this situation we get the relation 
$P_{W/Y}(-E_Z) \cdot E_{\widetilde{V_1}} \dots E_{\widetilde{V_k}}=0$.

\item
Let $Z$ be a blow-up center of $Y$ with a Chern Polynomial $P_{Z/Y}$. 
The class $E_Z$ satisfies the equation $P_{Z/Y}(-E_Z)=0$.
\end{enumerate}

\subsection{Standard monomials and intersection pairings}

The relations described above can be used to obtain a smaller set of generators for tautological groups. 
These generators will be called \emph{Standard monomials}.
We will see that there is an involution in the tautological ring of $X[n]$ 
which gives a one to one correspondence between standard monomials in degrees $d$ and $n-d$.
The intersection pairing 
$$R^d(X[n]) \times R^{n-d}(X[n]) \too \QQ, \D 0 \leq d \leq n$$
has a simple form with respect to this choice of basis for tautological groups.
The existence of a natural filtration on the tautological algebra shows that the resulting intersection matrix becomes triangular.
It consists of square blocks along the main diagonal.
These blocks are intersection matrices of pairings for $X^m$ when $m \leq n$.
This also shows that to study tautological classes on $X[n]$ 
it is enough to restrict to the spaces $X^m$, $m \leq n$.

Every element in the tautological ring is a product of tautological classes on $X^n$ with products of exceptional divisors.
It can be written as follows:

\begin{equation}\label{V}\tag{4}
v:=a(v) \cdot b(v) \cdot \prod_{r=1}^m D_{I_i}^{i_r},
\end{equation}

where $a(v)$, $b(v)$ belong to $R^*(X^n)$ defined before and 
$D_I$ is the exceptional divisor associated with the subvariety $X_I$. 
To define standard monomials and the involution on the tautological ring 
we need to associate a graph to each monomial in $R^*(X[n])$. 
Let $v$ be a monomial as given in \eqref{V}.
We associate a directed graph $\G:=(V_{\G},E_{\G})$ to $v$ as follows:
The vertex set $V_{\G}$ of $\G$ is identified with the set $\{1, \dots, m\}$.
There is an edge $(r,s)$ in $E_{\G}$ from the vertex $r$ to $s$ if the set $I_s$ is a maximal element of the set
$$\{I_i: I_i \subset I_r\}.$$
The degree deg$(i)$ of a vertex $i$ is the number of outgoing edges $(i,j)$ with the starting point $i$.
The closure $\overline{i}$ of a vertex $i \in V_{\G}$ consists of all vertices $r$ such that the inclusion $I_r \subseteq I_i$ holds.
In other words, the vertex $r$ belongs to the closure of $i$ when there is a path in the graph $\G$ from $i$ to $r$.
Notice that each vertex $i$ of $\G$ corresponds to a subset $I_i$ of $\{1, \dots, n\}$.  
A vertex $i$ is called a \emph{root} of $\G$ if the set $I_i$ is a maximal element of the collection 
$\{I_1, \dots, I_m\}$ with respect to inclusion of sets.   
A vertex $i$ corresponding to a minimal subset $I_i$ is called an \emph{external} vertex.
All the other vertices of $\G$ are called \emph{internal} vertices.

Let $v$ be a monomial in the tautological ring of $X[n]$ and denote by $\G$ the associated graph.
Roots of $\G$ correspond to the collection $\{J_1, \dots, J_s\}$ consisting of subsets of $\{1, \dots, n\}$.
For each $1 \leq r \leq s$ let $\alpha_r \in J_r$ be the smallest element. 
Consider the following subset $S$ of the set $\{1, \dots, n\}$:

\begin{equation} \label{S} \tag{5}
S:=\{\alpha_1, \dots, \alpha_s\} \cup ( \cap_{r=1}^m I_r^c).
\end{equation}

In this situation we say that the monomial $v$ is standard if 
$$a(v)b(v) \in R^*(X^{|S|})$$
is a standard monomial according to our former convention \eqref{STC} in Section 3.3.
We also require the following restriction for the powers of exceptional divisors:
$$i_r \leq \text{min}(|I_r|-2,|I_r|-|\cup_{I_s \subset I_r}I_s|+\text{deg}(I_r)-2).$$

The following proposition shows that in studying tautological classes on $X[n]$ we can restrict to standard monomials:

\begin{prop}\label{standard}
Standard monomials form an additive basis for tautological groups in all degrees.
\end{prop}

\begin{proof}
Let $v$ be a monomial as in Definition \eqref{V}.
For any subset $I$ of the set $\{1, \dots, n\}$
and elements $i,j \in I$ the relations 
$$a_i \cdot D_I=a_j \cdot D_I,  \D b_{i,j} \cdot D_I=-2g a_i \cdot D_I$$ 
hold. 
Under these assumptions for every $k \in \{1, \dots , n\}$
we have that $b_{i,k} \cdot D_I=b_{j,k} \cdot D_I$.
Using these relations we may assume that 
the monomial $a(v)b(v)$ belongs to $R^*(X^{|S|})$ for the index set $S$ given in \eqref{S}.
We may also assume that the power $i_r$ satisfies the inequality 
$i_r \leq |I_r|-2$.
This follows from the last class of relations in \ref{relations}. 
To deal with the last inequality let $\{J_1, \dots, J_s\}$ be the set of maximal elements of the set
$$\{I_i: I_i \subset I_r, \text{where} \ 1 \leq i \leq m\}.$$
From the relations of type (4) in \ref{relations} the monomial 
$D_{I_r}^j \cdot \prod_{i=1}^s D_{J_i}$
can be written as a sum of terms which are strictly less than it,
where $$j=|I_r|-|\cup_{i=1}^s J_i|+s-1=|I_r|-|\cup_{I_s \subset I_r} I_s|+\deg(I_r)-1.$$
This shows that for any $r$ we may assume that 
$$i_r \leq |I_r|-|\cup_{I_s \subset I_r} I_s|+\deg(I_r)-2.$$
\end{proof}

We are now able to define an involution on $R^*(X[n])$ which associates a dual element $v^* \in R^{n-d}(X[n])$ to every standard monomial $v \in R^d(X[n])$.

\begin{dfn}
Let $v$ be a standard monomial as given in \eqref{V} and $\G$ be its associated graph.
Define the subsets $J_1, \dots, J_s$ and $S$ of $\{1, \dots, n\}$ as before.
The set $T$ is defined as
$$T:=S \setminus A_v \cup B_v.$$
The integers $j_r$ for $1 \leq r \leq m$ are defined as follows:

$$j_r:= \left\{ \begin{array}{ll}
|I_r|-|\cup_{I_s \subset I_r}I_s|+\deg(I_r)-1-i_r & \qquad I_r \ \mathrm{is \ an \ internal \ vertex \ of} \ \mathcal{G} \\ \\
|I_r|-1-i_r & \qquad I_r \ \mathrm{is \ an \ external \ vertex \ of} \ \mathcal{G}.  \\
\end{array} \right. $$

We define 
$$v^*:=a(v^*) \cdot b(v^*) \cdot \prod_{r=1}^m D_{I_i}^{j_r},$$
where $a(v^*)=\prod_{i \in T}a_i$ and $b(v^*)=b(v)$.

\end{dfn}

It follows from our definition that the dual of every standard monomial is again standard.
The basic property $v^{**}=v$ shows that this association defines an involution on the tautological algebra. 

\begin{ex}
Let $I$ be a subset of $\{1, \dots, n\}$ with at least 3 elements
and denote by $i$ its smallest element.
The divisor $v:=D_I$ is standard.
Its dual is $v^*=a_iD_I^{n-2}$. 
The intersection product $v \cdot v^*$ is equal to $(-1)^{n-1}$.
\end{ex}

\begin{ex}
Let $I_1, \dots, I_4$ be 4 disjoint subsets of $\{1, \dots ,n\}$ having at least 3 elements.
Let $i_j \in I_j$ be the smallest element for $1 \leq j \leq 4$.
The degree 6 class $v:=b_{i_1,i_2}b_{i_3,i_4} \prod_{i=1}^4D_{I_i}$ is standard.
We have that $v^*=b_{i_1,i_2}b_{i_3,i_4} \prod_{i=1}^4D_{I_i}^{|I_i|-2}$
and $v \cdot v^*=4(-1)^{1+\sum_{i=1}^4 |I_i|}g^2$.
\end{ex}

\begin{ex}
Let $n=20$ and consider the monomial $v=b_{1,2} \cdot \prod_{i=1}^7 D_{I_i}$
in $R^8(X[20])$, where
$$I_1=\{3,4,5\}, \D I_2=\{6,7,8\}, \D I_3=\{9,10,11\}, \D I_4=\{3,4, \dots, 13\},$$
$$I_5=\{14,15,16\}, \D I_6=\{14,15,16,17,18\}, \D I_7=\{14,15, \dots, 20\}.$$

It follows from the definition that $v$ is a standard monomials.
The graph $\G$ associated to $v$ is pictured below.
It has 7 vertices and 5 edges. 
The vertex $I_1$ has degree 3, the vertices $I_5,I_6$ are of degree 2 and other vertices have degree zero.
The graph $\G$ has two roots $I_1$ and $I_5$.
There are 4 external vertices $I_2,I_3,I_4$ and $I_7$ and 3 internal vertices $I_1,I_5,I_6$.
The dual of $v \in R^{12}(X[20])$ is defined as 
$$v^*=b_{1,2}a_3a_{14} \cdot \prod_{r=1}^7 D_{I_r}^{i_r},$$
where $i_4=3$ and all other powers are 1. 
The intersection product $v \cdot v^*$ is equal to $-2g$.

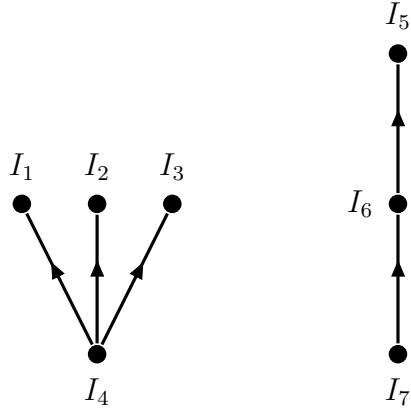
\begin{figure}[htp]
\begin{tikzpicture}
[inner sep=0pt,thick,
    dot/.style={fill=black,circle,minimum size=7pt}]
  [auto=left]
\begin{scope}[very thick, every node/.style={sloped,allow upside down}]
\node (m1) at (2,7.5) {$I_4$};
\node[dot] (n1) at (2,8) {};
\node (m2) at (1,10.5)  {$I_1$};
\node[dot] (n2) at (1,10) {};
\node (m3) at (2,10.5) {$I_2$};
\node[dot] (n3) at (2,10) {};
\node (m4) at (3,10.5) {$I_3$};
\node[dot] (n4) at (3,10) {};
\draw (n1)-- node {\midarrow} (n2);
\draw (n1)-- node {\midarrow} (n3);
\draw (n1)-- node {\midarrow} (n4);
\node[dot] (n5) at (6,8) {};
\node (m5) at (6,7.5) {$I_7$};
\node[dot] (n6) at (6,10) {};
\node (m6) at (5.5,10)  {$I_6$};
\node[dot] (n7) at (6,12) {};
\node (m7) at (6,12.5) {$I_5$};
\draw (n5)-- node {\midarrow} (n6);
\draw (n6)-- node {\midarrow} (n7);
\end{scope}
\end{tikzpicture}
\caption{The graph $\mathcal{G}$}
\end{figure}

\end{ex}

\subsection{The filtration of the tautological algebra}

To compute the intersection matrix of the pairing we consider all standard monomials of degree $d$ and their duals defined before.
The resulting matrix has a triangular structure which simplifies our analysis of the pairing. 
The triangular property is shown to be connected to the vanishing of certain intersection numbers.
These vanishings are better understood via a natural filtration of the tautological ring.
The existence of this filtration was predicted by Looijenga in the previous work \cite{T1} on the tautological ring of $M_{1,n}^{ct}$.
This filtration is defined in terms of a natural ordering of monomials.
We first order the exceptional divisors.

\begin{dfn}\label{<}
Let $I,J$ be subsets of the set $\{1, \dots, n\}$. We say that $I < J$ if 
\begin{itemize}
\item
$|I| < |J|$;

\item
or if $|I|=|J|$ and the smallest element of $I \setminus I \cap J$ is smaller than the smallest element of $J \setminus I \cap J$.
\end{itemize}
This induces an ordering on monomials in $R^*(X[n])$. 
Consider two monomials $v_1,v_2$ in the tautological ring of $X[n]$ as follows:
$$v_1=a(v_1)b(v_1) \cdot \prod_{r=1}^{r_0} D_{I_r}^{i_r} \cdot D, \D
v_2=a(v_2)b(v_2) \cdot \prod_{r=1}^{r_0} D_{I_r}^{j_r} \cdot D,$$
where $D=\prod_{r=r_0+1}^m D_{I_r}^{i_r}$,
for $I_m < \dots < I_1$ and $i_{r_0}<j_{r_0}$;
or if $r_0=0$ and $a(v_1)b(v_1)<a(v_2)b(v_2)$.
Furthermore, we say that $v_1 \ll v_2$ if for any factor $D_I$ of $v_2$ we have that $v_1<D_I$. Notice that $v_1 \ll v_2$ implies that $v_1<v_2$.
\end{dfn}

This ordering of monomials induces a filtration of the tautological ring.

\begin{dfn}\label{fil}
Let $v$ be a standard monomial as given in \eqref{V} and let $J_1, \dots, J_s$ be the roots of its associated graph as before.
The integer $p(v)$ is defined to be the degree of the element
$$a(v)b(v) \cap_{r=1}^s X_{J_r} \in A^*(X^n),$$
which is the same as 
$$\deg a(v)b(v) +\sum_{r=1}^s |J_r|-s.$$
The subspace $F^pR^*(X[n])$ of the tautological ring is defined to be the $\QQ$-vector space generated by standard monomials $v$ satisfying $p(v) \geq p$.
\end{dfn}

It follows immediately from our definition that
$$F^{p+1}R^*(X[n]) \subseteq F^pR^*(X[n]).$$
This shows that the subspaces $F^pR^*(X[n])$ indeed define a filtration for $R^*(X[n])$.
The following vanishing result will be crucial in the analysis of 
intersection pairings:

\begin{prop} \label{3}
\begin{enumerate}
\item 
Let $v \in F^pR^*(X[n])$ and $w \in R^d(X[n])$
be such that $w \ll v$. 
If $p+d>n$ then the intersection product $v \cdot w$ is zero. 
In particular, we have that $F^{n+1}R^*(X[n])$ is zero. 

\item
Let $v_1,v_2$ be standard monomials in $R^d(X[n])$ satisfying $D(v_1) < D(v_2)$.
Then $v_1 \cdot v_2^*=0$. 
\end{enumerate}
\end{prop}
\begin{proof}
Let $J_1, \dots ,J_s$ be as in Definition \ref{fil}.
Denote by $Y$ the space before blowing up the subvarieties $X_{J_1}, \dots , X_{J_s}$ and by $\widetilde{Y}$ the resulting blow-up space.
The corresponding morphism is denoted by 
$$\pi: \widetilde{Y} \rightarrow Y.$$
The product $v \cdot w$ vanishes since it is the pull-back of zero via $\pi^*$. 
This proves the first claim.

For the second part it is enough to write $v_1 \cdot v_2^*$ as a product $v \cdot w$, for $v,w \in R^*(X[n])$ satisfying the properties stated in the first part.
To find $v$ and $w$, let $v_1,v_2$ be given as in Definition \ref{<}, and denote by $\{J_1, \dots ,J_s\}$ the set of roots of the graph associated to the monomial 
$$D=\prod_{r=r_0+1}^{m} D_{I_r}^{i_r}.$$

By relabelling the roots we may assume that there is an $s_0 \geq 0$ such that $J_r \subset I_{r_0}$ for $1 \leq r \leq s_0$, 
and the equality $I_{r_0} \cap J_r=\emptyset$ holds for $s_0 < r \leq s$. 

Let $w$ be the product of all monomials $D_I$ in $v_1 \cdot v_2^*$ which are strictly less than $D_{I_{r_0}}$ and $v$ 
be the product of the other factors, so that the equality 
$v_1 \cdot v_2^*=v \cdot w$ holds.
Notice that $w \ll v$, by the definition of $v$ and $w$. 
We need to calculate the degree of $w$.
This is done by calculating the degree of $v$ and using the equation $$\deg(w)=n-\deg(v).$$

Roots of the graph associated to $v$ consist of 
$s-s_0+1$ vertices associated to the subsets 
$I_{r_0}, J_{s_0+1}, \dots, J_s$.
The degree $\deg(v)$ of $v$ is therefore equal to the following sum:
$$|I_{r_0}|-1+\sum_{r=s_0+1}^s (|J_r|-1).$$

It follows that
$$\deg(w)=n+1+j_{r_0}-i_{r_0}+s-s_0-|I_{r_0}|-\sum_{r=s_0+1}^s |J_r|
> n+1+s-s_0-|I_{r_0}|-\sum_{r=s_0+1}^s |J_r|.$$ 
On the other hand from Definition \ref{fil} we have that
$$p(v) \geq |I_{r_0}|+\sum_{r=s_0+1}^s |J_r|-s+s_0-1.$$
From the inequality $\deg(w)+p(v)> n$ we see that the intersection product $v \cdot w$ is zero.
\end{proof}

\begin{thm}\label{X[n]}
Let $X$ be a hyperelliptic curve of genus $g$ and $n$ be an integer.
Denote by $R^*(X[n])$ its tautological ring.
The intersection pairings
$$R^d(X[n]) \times R^{n-d}(X[n]) \too \QQ$$
are perfect for all $0 \leq d \leq n$.
The space of relations is generated by the vanishing of the following cycles:

\begin{itemize}

\item
The Faber-Pandharipande cycle,

\item
The Gross-Schoen cycle,

\item 
All product of the form $D_I \cdot D_J$ for subsets $I,J$ of $\{1, \dots, n\}$
whenever $I \not \subseteq J$, $J \not \subseteq I$ and $I \cap J \neq \emptyset$.

\item
All products of the form $x \cdot D_I$,
where $I$ is a subset of $\{1, \dots , n\}$ with at least 3 elements and
$x$ has the form 
$d_{i,j}+K_j$ or $d_{i,k}-d_{j,k}$ 
for $i,j \in I$ and $k \in \{1, \dots , n\} \setminus I$.

\end{itemize}

\end{thm}
\begin{proof}
We have proved in Proposition \ref{standard} that tautological groups are generated by standard monomials.
Consider a basis for the vector space $R^d(X[n])$ consisting of standard monomials of degree $d$. 

It follows from Proposition \ref{3} that the intersection matrix of the pairing between tautological classes has a triangular structure with respect to our ordering of generators.
To finish the study of the pairing we need to analyze the blocks along the main diagonal.
We will see that these are intersection matrices of the pairings of 
$R^*(X^{|S|})$
for various subsets $S$ of $\{1, \dots, n\}$.
In other words we need to understand the block corresponding to standard monomials having the same $D$-part.
Let $D$ be any such monomial.
It is a product of the divisor classes $D_I$.
Consider the graph $\G$ associated to the monomial $D$.
The subset $S$ of the set $\{1, \dots, n\}$ is defined as in \ref{S}.
Now suppose that $v_1,v_2$ are standard monomials satisfying $D(v_1)=D(v_2)$. 
It is easy to see that the intersection numbers
$$v_1 \cdot v_2^* \in R^n(X[n]) \cong \QQ$$
and the number
$$a(v_1)b(v_1) \cdot a(v_2^*)b(v_2^*) \in R^{|S|}(X^{|S|}) \cong \QQ$$
differ by $(-1)^{\epsilon}$, where 
$$\epsilon=|\cup_{r=1}^m I_r|+\sum_{i \in V(\G)} \deg(i).$$

This shows that square blocks on the diagonal of the intersection matrix are related to the pairings of $R^*(X^{|S|})$ for various $S$ in a simple way. 
In particular, we see the direct connection between tautological relations on $X[n]$ and the products $X^{|S|}$.
We have studied the pairings for $R^*(X^n)$ for all $n$ and we have found all tautological relations.
This shows the same result for $X[n]$ as well.

We need to prove that the relations stated above form a set of generators.
It is not clear that all relations described in \ref{relations} follow from these relations.
The argument in the next section shows this.
\end{proof}

\section{The tautological ring of $\CH_{g,n}^{rt}$}

We will show that the tautological ring of the moduli space $\CH_{g,n}^{rt}$ has the same description as $R^*(X[n])$
studied in the previous section.
Let $X$ be a smooth hyperelliptic curve of genus $g$ as before.
Consider the natural map
$$F:X[n] \too \CH_{g,n}^{rt}.$$
Tautological classes on both sides are connected to each other via the following relations:
$$F^*(K_i)=K_i \  \text{for} \ 1 \leq i \leq n,$$
$$F^*(D_{i,j})=d_{i,j}-\sum_{i,j \in I}D_I, \D
F^*(D_I)=D_I \ \text{when} \ |I| \geq 3.$$

As a result we get a ring homomorphism between tautological algebras:
$$F^*: R^*(\CH_{g,n}^{rt}) \too R^*(X[n]).$$
We want to prove that this homomorphism identifies $R^*(\CH_{g,n}^{rt})$
with the tautological ring of the fiber $X[n]$.
It is clear from the definition that $F^*$ is surjective.
We also need to show the injectivity of this map.
Notice that both rings are generated by divisors and there is a natural bijection between these generators.
To prove the injectivity it is enough to verify that for every relation among divisors on the fiber $X[n]$ the corresponding divisors on $\CH_{g,n}^{rt}$ satisfy the same relation.

\begin{enumerate}

\item
Our arguments in sections 2 and 3 show that all tautological relations on the fiber $X^n$ hold on the moduli space $\CC^n$.
Their pull back to $\CH_{g,n}^{rt}$ via the contraction map
$$\CH_{g,n}^{rt} \too \CC^n$$
establishes the first class of relations.

\item We now consider the first class of relations among the divisor classes $D_I$. 
The product $D_I \cdot D_J \in R^2(\CH_{g,n}^{rt})$ is zero unless 
$$I \subseteq J, \mathrm{or} \qquad  J \subseteq I, \mathrm{or} \qquad I \cap J = \emptyset.$$
Notice that this holds even in $R^*(\overline{\CH}_{g,n})$ for all subsets $I,J$ with at least two elements. 
In (2) we only consider subsets with at least three elements. 
 
\item
Let $I \subset \{1,\dots,n\}$ be a subset with $|I| \geq 3$ and $i_I:X_I \rightarrow X^n$ denote the inclusion. 
The $\ker(i_I^*)$ is generated by the divisor classes 
$d_{i,j}+K_i$ and $d_{i,k}-d_{j,k}$ 
for distinct elements $i,j \in I$ 
and $k$ in the complement of $I$.  
We will prove that 
$$( \psi_i + \sum_{i \in J, j \notin J} D_J) \cdot D_I \D
\text{and} 
 \ (\sum_{i,j \in J} D_J - \sum_{i,k \in J} D_J) \cdot D_I$$ 
are zero in $R^2(\CH_{g,n}^{rt})$. 
The first equality follows from the well-known formula for the $\psi$ classes in genus zero, which we recall: 
Let $i \in \{1, \dots , n\}$ be an element and assume that $j,k \in \{1, \dots, n\} \backslash \{i\}$ are arbitrary distinct elements. Then one has the following equality in 
$A^1(\overline{M}_{0,n})$: $$\psi_i=\sum_{\substack{i \in I \\j,k \notin I}}D_I.$$ 
The second relation is an easy implication of the relations proved in the previous part.  

\item Let $V=V_1 \cap \dots \cap V_k,W$ and $Z$ be subvarieties of $X^n$ as in part (4) in \ref{relations}, so that $V \cap W=Z$. 
After possibly relabeling the indices, we can assume that 
$$Z=X_{I_0}, \qquad V_i=X_{I_i}, \ \mathrm{for} \ 1 \leq i \leq k, \qquad W=\prod_{i=2}^{r_1} d_{1,i} \cdot \prod_{j=1}^k d_{1,r_j+1} ,$$
where $1 \leq r_1 < \dots < r_{k+1} \leq n, I_0=\{1, \dots ,r_{k+1}\}$, and $I_i=\{r_i+1, \dots , r_{i+1}\}$ for $1 \leq i \leq k$. 
From these data we get the relation 
$$P_{W/X^n}(-\sum_{I_0 \subseteq I} D_I) \cdot \prod_{i=1}^k D_{I_i}=0$$ on the space $X[n]$,
for $$P_{W/X^n}(t)=\prod_{i=2}^{r_1} (t+d_{1,i}) \cdot \prod_{j=1}^k (t+d_{1,r_j+1}).$$ 
We want to prove a similar identity on the moduli space $\CH_{g,n}^{rt}$.
This is proven by showing that any monomial in the expansion of this expression is zero. 
Consider any such monomial. 
It has the form
$$\prod_{i=2}^{r_1} D_{J_i} \cdot \prod_{j=1}^{k} D_{J_{r_j+1}} \cdot D_{I_j},$$

where

$$1,i \in J_i \ \text{for} \ 2 \leq i \leq r_1, \D 1,r_{j+1} \in J_{r_i+1} \ \text{for} \ 1 \leq j \leq k.$$

Assume that the product above is not zero.
Notice that the subsets $J_i$ are not disjoint since 1 is in their intersection.
This means that for any two indices $i$ and $j$ one of the inclusions 
$J_i \subseteq J_j$ or $J_j \subseteq J_i$ holds. 
When $1 \leq j \leq k$ the non vanishing of the product $D_{J_{r_j+1}} \cdot D_{I_j}$ implies that 
$I_j  \subseteq J_{r_j+1}$. 
This means that 
$\cup_{j=1}^{k}I_j \subseteq \cup_{j=1}^k J_{r_j+1}$.
Since $1,i \in J_i$ for $2 \leq i \leq r_1$ we see that the union of all $J_i$'s contains $I_0$.
But this union coincides with one of the $J_i$'s.
This means that $I_0 \subseteq J_i$ for some $i$.
But the term $D_{J_i}$ for such set $J_i$ is excluded from the expression above.
This contradiction shows the desired vanishing.

\item Let $I \subset \{1,\dots,n\}$ be a subset with $|I| \geq 3$ corresponding to the subvariety $Z=X_I$ of $X^n$. The relation 
$$P_{Z/X^n}(-\sum_{I \subseteq J}D_J)=0$$ on the space $X[n]$ holds. 
To prove the corresponding relation on the moduli space we will show that the product $$\prod_{i \neq j \in I} (d_{i,j}- \sum_{I \subseteq J} D_J)$$ is zero, 
where $i \in I$ is an arbitrary element.
This is verified as in the previous case by showing that all monomials occurring in the expansion of the expression above are zero. 
In the proof we only need to use relations of type (2).
\end{enumerate}

Notice that our proof shows that seemingly complicated relations of type (4) and (5) are formal consequences of the trivial relations of type (2).
This completes the proof of Theorem \ref{X[n]} as well.

The argument above proves that $F^*$ is indeed an isomorphism. 
This shows that the restriction map induces an isomorphism between the tautological rings involved and finishes the proof of Theorem \ref{X}.

\section{Tautological cohomology via monodromy}  

We have seen a complete description of
relations among tautological classes in the Chow ring of $\CH_{g,n}^{rt}$.
From the Gorenstein property of the tautological ring there will be no more relations between tautological cycles in cohomology.
More precisely, consider the cycle class map
$$\text{cl}: A^*(\CH_{g,n}^{rt}) \too H^*(\CH_{g,n}^{rt}).$$

The image of $R^*(\CH_{g,n}^{rt})$ is called the tautological cohomology of $\CH_{g,n}^{rt}$ and is denoted by $RH^*(\CH_{g,n}^{rt})$.
From Theorem \ref{X} we get the following isomorphism:
$$\text{cl}: R^*(\CH_{g,n}^{rt}) \stackrel{\sim}{\rightarrow} RH^*(\CH_{g,n}^{rt}).$$ 

This shows that the tautological cohomology ring $RH^*(\CH_{g,n}^{rt})$ has Poincar\'e duality as well.
Now let $X$ be a smooth hyperelliptic curve of genus $g$.
The tautological cohomology $RH^*(X[n]) \subset H^*(X[n])$ is defined analogously. 
The cycle class map identifies the tautological algebras of 
$X[n]$ in Chow and cohomology.
We therefore get the following diagram of isomorphisms from Theorems \ref{X} and \ref{X[n]}:

\begin{equation}\label{square}\tag{6}
\begin{CD}
R^*(\CH_{g,n}^{rt}) @>F^*>> R^*(X[n]) \\
@VV\text{cl}V @VV\text{cl}V\\
RH^*(\CH_{g,n}^{rt}) @>F^*>> RH^*(X[n])
\end{CD}
\end{equation}

We want to identify these isomorphic algebras with monodromy invariant classes in the rational cohomology of $X[n]$. 
This follows from a result of Petersen and Tommasi proved for arbitrary curves which we will recall below.

Let $X$ be a smooth curve of genus $g$.
Recall that there is a natural action of the symplectic group $\mathfrak{Sp}(2g,\QQ)$ on the rational cohomology ring of $X$.
This induces an action of this group on the cohomology of $X[n]$ as well.
As it was observed in \cite{PT} the tautological cohomology of $X[n]$ can be identified with monodromy invariant classes in cohomology:

\begin{prop}
Let $X$ be a smooth projective curve of genus $g$.
The subalgebra $H^*(X[n])^{\mathfrak{Sp}(2g,\QQ)}$ of monodromy invariant classes coincides with the tautological cohomology ring $RH^*(X[n])$.
\end{prop}

In particular, $RH^*(X[n])$ is a Gorenstein algebra for every curve $X$ and a natural number $n$.
This is very different from what we know in Chow as $R^*(X^n)$ and $R^*(X[n])$ don't have Poincar\'e duality for generic curves already when $n=2$. However for hyperelliptic curves we proved the isomorphism between tautological rings in Chow and cohomology.
This shows that all isomorphic algebras in the diagram \ref{square} 
can be naturally identified with monodromy invariant classes.
This finishes the proof of Corollary \ref{M}. 

\section{A remark on Pixton's conjecture}

Our results in previous sections deal with relations among tautological classes on the space $\CH_{g,n}^{rt}$.
In this section we want to use the found relations to get relations on the space $\M_{g,n}$. This reformulation uses the following ingredients:
the formula for the class of hyperelliptic curves and the natural evaluation map on the space of curves of rational tails.
Then we raise questions about the connection between the resulting relations 
and Pixton's recent conjecture.

The class of the hyperelliptic locus inside $M_g$ is calculated by Mumford in the seminal paper \cite{M}. 
It is a class of degree $g-2$ and the formula is a polynomial in terms of lambda and kappa classes.
Recall that the lambda class $\la_i$ for $0 \leq i \leq g$ is defined as
the Chern class $\la_i:=c_i(\E)$ of the Hodge bundle $\E$.
On the other hand we know that the tautological group $R^{g-2}(M_g)$ has dimension one.
Indeed the lambda class $\la_{g-2}$ is a generator.
Faber \cite{F1} calculated the explicit multiplicity:

\begin{equation}\label{H}\tag{7}
[\CH_g]=\frac{2^{2g}-1}{g(g+1)(4g^2-1)|B_{2g-2}|}\la_{g-2} \in R^*(M_g),
\end{equation}

where $B_{2g-2}$ is the Bernoulli number.
Let us explain how we get relations on $\M_{g,n}$ from those on hyperelliptic curves.
We have seen that the vanishing of the Faber-Pandharipande cycle on $\CH_{g,2}^{rt}$ is equivalent to the relation 
$$K_1K_2=(2g-2)K_1d_{1,2}$$
where $d_{1,2}$ is the diagonal class.
We now use the formula for the class of the hyperelliptic locus. 
If we rewrite this relation in terms of standard tautological classes on $M_{g,2}^{rt}$ we get that
$$(\psi_1\psi_2-(2g-1)\psi_1^2) \cdot \la_{g-2}=0 \in R^g(M_{g,2}^{rt}).$$
To obtain a relation in $R^*(\M_{g,2})$ we use the evaluation 
$$A^*(M_{g,n}^{rt}) \too \QQ, \D \epsilon \too \int_{\M_{g,n}} \epsilon \cdot \la_{g-1} \la_g.$$
Recall that this map defines a well-defined evaluation on $A^*(M_{g,n}^{rt})$.
The resulting equality becomes
$$\int_{\M_{g,2}} (\psi_1\psi_2-(2g-1)\psi_1^2) \la_{g-2}\la_{g-1}\la_g=0.$$
This is a well-known identity among the simplest type of Hodge integrals.
It can be easily shown using the string equation.
However notice that this last identity does \emph{not} prove the vanishing of the Faber-Pandharipande cycle on $\CH_{g,2}^{rt}$.
We can apply the same method to other relations.
The vanishing of the Gross-Schoen cycle gives a relation of degree $3g-1$ on $\M_{g,3}$.
The resulting relation is slightly more complicated than the previous one.
On the space $\CH_{g,3}^{rt}$ it has the following form:
$$d_{1,2}d_{1,3}-\frac{1}{2g-2}(K_1d_{2,3}+K_2d_{1,3}+K_3d_{1,2})+\frac{1}{(2g-2)^2}(K_1K_2+K_1K_3+K_2K_3).$$
where,
$$d_{i,j}=D_{i,j}+D_{1,2,3}, \D K_i=\psi_i-\sum_{i \in I}D_I.$$
The divisor $D_I$ corresponds to curves with the marking set $I$ on the rational component. 
To obtain a relation on $\M_{g,3}$ we proceed as before.
The resulting relation seems less obvious to prove with known methods.
In the same way we get a tautological relation of degree $4g-2$ on $\M_{g,2g+2}$.
The resulting relation is symmetric with respect to the $2g+2$ points.
It has the following form:

\begin{thm}
The following relation holds in $R^{4g-2}(\M_{g,2g+2})$:
$$\left (\sum_{\mathcal{I}} b_{i_1,i_2} \dots b_{i_{2g+1},i_{2g+2}} \right) \cdot \la_{g-2}\la_{g-1}\la_g=0.$$
Each term of the sum corresponds to a partition $\mathcal{I}$ of $\{1, \dots, 2g+2\}$ into $g+1$ subsets with 2 elements. 
\end{thm}

In \cite{Pi} Pixton proposed a class of conjectural tautological relations on the Deligne-Mumford space $\M_{g,n}$.
His relations are defined as a weighted sum over all stable graphs on the boundary $\partial M_{g,n}=\M_{g,n} \setminus M_{g,n}$.
Weights on graphs are tautological classes.
Pixton also conjectures that all tautological relations have this form.
A recent result \cite{PPZ} of Pandharipande, Pixton and Zvonkine shows that Pixton's relations are connected with the Witten's class on the space of 3-spin curves.
Their analysis shows that Pixton's relations hold in cohomology.
Janda \cite{J} derives Pixton's relations by applying the virtual localization formula to the moduli space of stable quotients introduced in \cite{MOP}.
This establishes Pixton's relations in Chow. 
Several known relations such as Keel's relation \cite{K} on $\M_{0,4}$, 
Getzler's relation \cite{G1} on $\M_{1,4}$ and Belorousski-Pandharipande relation \cite{BP} on $\M_{2,3}$ follow from Pixton's relations.
While this method produces a large class of relations there are very few evidences for Pixton's conjecture at this point.
The mysterious structure of the tautological algebras and the complicated nature of the relations makes his conjecture more delicate.
The question whether all relations in the tautological ring come from Pixton's relations seems interesting and challenging. 
The relations we described in this section give an example of an infinite collection of tautological relations. 
The connection between these relations and Pixton's relations is not clear to us. 
We can show that the relation on $M_{2,6}^{rt}$ already
comes from a stable quotient relation on $\M_{2,6}$.
Our computation is based on the notes \cite{JSQ} by Janda on stable quotient relations on products of the universal curve.
We don't know if it is possible to derive our relations from Pixton's relations in general.
Understanding their connection needs new ideas.

\section{Concluding remarks and further directions}

A conjecture of Faber and Pandharipande \cite{FP2} predicts that every relation in 
$R^*(M_{g,n}^{rt})$ and $R^*(M_{g,n}^{ct})$ can be extended to a tautological
relation in $R^*(\M_{g,n})$.
It is natural to ask whether tautological relations on $\CH_{g,n}^{rt}$ can be extended to 
tautological relations on $\CH_{g,n}^{ct}$ and $\overline{\CH}_{g,n}$.
As we have seen in Section \ref{g+1} the degree $g+1$ relation on $\CH_{g,2g+2}^{rt}$
can be naturally extended to the space of curves of compact type.
The reason is that the Abel-Jacobi map extends to curves of compact type and the vanishing of
$\theta^{g+1}$ holds over the universal Jacobian over $\CH_g^{ct}$.
The formula of Grushevsky and Zakharov gives an explicit formula for the boundary terms arising in
the extended relation.
Working with semi-abelian varieties and using the result of \cite{GZ2} 
should prove the existence of an extension of this relation to $\overline{\CH}_{g,n}$ as well.
Extending the Faber-Pandharipande and Gross-Schoen relations to the compactifications is more subtle.
A naive guess would be to obtain these vanishings from the degree $g+1$ relation.
One could start from this relation and push it forward to spaces with smaller numbers of points.
At each step we can multiply the resulting relation with a tautological class.
In genus two both of these relations follow from the vanishing of $\theta^3$ after suitable 
push-forwards. 
But in genus 3 one can only obtain the Faber-Pandharipande relation with this method.
In genus $g \geq 4$ none of these relations follow from this approach.
In fact if this method gives the vanishing of these cycles the same argument would prove the same
statement for \emph{every family} of curves with trivial kappa classes in positive degrees.
In particular, it would show their triviality for any fixed curve.
But according to \cite{Y1} the Faber-Pandharipande cycle does not vanish for a generic curve of genus $g \geq 4$.
We also know the non triviality of the Gorss-Schoen cycle for generic curves of 
genus $g \geq 3$.
This is shown by Ceresa \cite{CE} in characteristic zero and by Fakhruddin \cite{FA} in positive characteristic.

Finding a complete description of tautological relations on $\CH_{g,n}^{ct}$ and 
$\overline{\CH}_{g,n}$ seems an interesting question.
In \cite{PT} Petersen and Tommasi show the existence of a counterexample to the Gorenstein conjectures
for $\M_{2,n}$ for some $n \leq 20$.
A more recent result \cite{P2} of Petersen shows that the tautological ring of  
$M_{2,n}^{ct}$ is not Gorenstein when $n \geq 8$.
These results suggest that the structure of the tautological rings of the larger compactifications
become more complicated already in genus two.
It is not clear whether for a fixed genus all relations can be obtained from finitely many ones independent of $n$.

Another natural question concerns other loci of curves with special linear systems.
The next case is the space of trigonal curves.
In \cite{T3} we study the connection between the \emph{modified Gross-Schoen cycles} and 
Faber's relations in \cite{F2} on certain spaces of trigonal curves.

\bibliographystyle{amsplain}
\bibliography{mybibliography}
\end{document}